\documentclass[preprint,12pt]{elsarticle}

\usepackage{amssymb}
 \usepackage{graphicx}
\usepackage{epsfig}

\usepackage{amsthm}

\newcommand{\sysn}{\left\{\begin{array}{rcl}}
\newcommand{\sysk}{\end{array}\right.}

\newcommand{\ingrw}[2]{\includegraphics[width=#1mm]{#2}}

\newtheorem{theorem}{Theorem}[section]
\newtheorem{lemma}[theorem]{Lemma}

\theoremstyle{example}

\theoremstyle{definition}
\newtheorem{definition}[theorem]{Definition}

\newtheorem{corollary}[theorem]{Corollary}

\journal{Topology and its Applications}

\begin{document}

\begin{frontmatter}

\title{The functional characterizations of the
Rothberger and Menger properties}


\author{Alexander V. Osipov}

\ead{OAB@list.ru}


\address{Krasovskii Institute of Mathematics and Mechanics, Ural Federal
 University,

 Ural State University of Economics, Yekaterinburg, Russia}

\begin{abstract} For a Tychonoff space $X$, we denote by $C_p(X)$
the space of all real-valued continuous functions on $X$ with the
topology of pointwise convergence. In this paper we continue to
study different selectors for sequences of dense sets of $C_p(X)$
started to study in the paper \cite{os2}.

A set $A\subseteq C_p(X)$ will be called {\it $1$-dense} in
$C_p(X)$, if for each $x\in X$ and an open set $W$ in $\mathbb{R}$
there is $f\in A$ such that $f(x)\in W$.

We give the characterizations of selection principles
$S_{1}(\mathcal{A},\mathcal{A})$,
$S_{fin}(\mathcal{A},\mathcal{A})$ and
$S_{1}(\mathcal{S},\mathcal{A})$  where

$\bullet$ $\mathcal{A}$ --- the family of $1$-dense subsets of
$C_p(X)$;

$\bullet$ $\mathcal{S}$ --- the family of sequentially dense
subsets  of $C_p(X)$, through the selection principles of a space
$X$. In particular, we give the functional characterizations of
the Rothberger and Menger properties.


\end{abstract}


\end{frontmatter}



\section{Introduction}

Throughout this paper, all spaces are assumed to be Tychonoff. The
set of positive integers is denoted by $\mathbb{N}$. Let
$\mathbb{R}$ be the real line, we put $\mathbb{I}=[0,1]\subset
\mathbb{R}$, and $\mathbb{Q}$ be the rational numbers. For a space
$X$, we denote by $C_p(X)$ the space of all real-valued continuous
functions on $X$ with the topology of pointwise convergence. The
symbol $\bf{0}$ stands for the constant function to $0$.

Basic open sets of $C_p(X)$ are of the form

$[x_1,...,x_k, U_1,...,U_k]=\{f\in C(X): f(x_i)\in U_i$,
$i=1,...,k\}$, where each $x_i\in X$ and each $U_i$ is a non-empty
open subset of $\mathbb{R}$. Sometimes we will write the basic
neighborhood of the point $f$ as $\langle f,A,\epsilon \rangle$
where $\langle f,A,\epsilon \rangle:=\{g\in C(X):
|f(x)-g(x)|<\epsilon$ $\forall x\in A\}$, $A$ is a finite subset
of $X$ and $\epsilon>0$.

 If $X$ is a space and $A\subseteq X$, then the sequential closure of $A$,
 denoted by $[A]_{seq}$, is the set of all limits of sequences
 from $A$. A set $D\subseteq X$ is said to be sequentially dense
 if $X=[D]_{seq}$. A space $X$ is called sequentially separable if
 it has a countable sequentially dense set.

In this paper, by a cover we mean a nontrivial one, that is,
$\mathcal{U}$ is a cover of $X$ if $X=\bigcup \mathcal{U}$ and
$X\notin \mathcal{U}$.

 An open cover $\mathcal{U}$ of a space $X$ is:

 $\bullet$ an {\it $\omega$-cover} if every finite subset of $X$ is contained in a
 member of $\mathcal{U}$.

$\bullet$ a {\it $\gamma$-cover} if it is infinite and each $x\in
X$ belongs to all but finitely many elements of $\mathcal{U}$.
Note that every $\gamma$-cover contains a countably
$\gamma$-cover.

For a topological space $X$ we denote:

$\bullet$ $\mathcal{O}$ --- the family of open covers of $X$;


$\bullet$ $\Gamma$ --- the family of countable open
$\gamma$-covers of $X$;




$\bullet$ $\Omega$ --- the family of open $\omega$-covers of $X$;



$\bullet$ $\mathcal{D}$ --- the family of dense subsets of
$C_p(X)$;

$\bullet$ $\mathcal{S}$ --- the family of sequentially dense
subsets of $C_p(X)$.

\bigskip

Many topological properties are defined or characterized in terms
 of the following classical selection principles.
 Let $\mathcal{A}$ and $\mathcal{B}$ be sets consisting of
families of subsets of an infinite set $X$. Then:

$S_{1}(\mathcal{A},\mathcal{B})$ is the selection hypothesis: for
each sequence $(A_{n}: n\in \mathbb{N})$ of elements of
$\mathcal{A}$ there is a sequence $\{b_{n}\}_{n\in \mathbb{N}}$
such that for each $n$, $b_{n}\in A_{n}$, and $\{b_{n}:
n\in\mathbb{N} \}$ is an element of $\mathcal{B}$.

$S_{fin}(\mathcal{A},\mathcal{B})$ is the selection hypothesis:
for each sequence $(A_{n}: n\in \mathbb{N})$ of elements of
$\mathcal{A}$ there is a sequence $\{B_{n}\}_{n\in \mathbb{N}}$ of
finite sets such that for each $n$, $B_{n}\subseteq A_{n}$, and
$\bigcup_{n\in\mathbb{N}}B_{n}\in\mathcal{B}$.

$U_{fin}(\mathcal{A},\mathcal{B})$ is the selection hypothesis:
whenever $\mathcal{U}_1$, $\mathcal{U}_2, ... \in \mathcal{A}$ and
none contains a finite subcover, there are finite sets
$\mathcal{F}_n\subseteq \mathcal{U}_n$, $n\in \mathbb{N}$, such
that $\{\bigcup \mathcal{F}_n : n\in \mathbb{N}\}\in \mathcal{B}$.

\medskip

Many equivalences hold among these properties, and the surviving
ones appear in the following Diagram (where an arrow denotes
implication), to which no arrow can be added except perhaps from
$U_{fin}(\Gamma, \Gamma)$ or $U_{fin}(\Gamma, \Omega)$ to
$S_{fin}(\Gamma, \Omega)$ \cite{jmss}.

\medskip

\begin{center}
\ingrw{90}{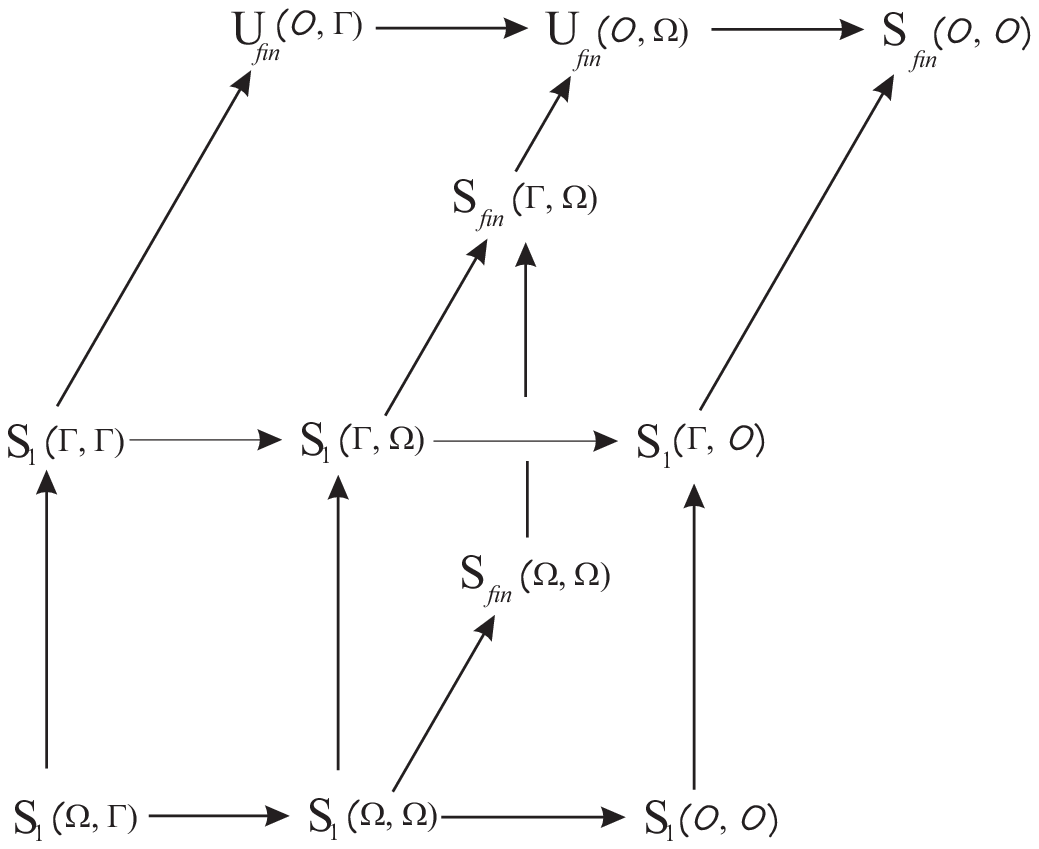}

\medskip

Fig.~1. The Scheepers Diagram for Lindel$\ddot{o}$f spaces.

\end{center}
\bigskip

The papers \cite{jmss,ko,sch3,sch1,bts} have initiated the
simultaneous
 consideration of these properties in the case where $\mathcal{A}$ and
 $\mathcal{B}$ are important families of open covers of a
 topological space $X$.

In papers
\cite{arh,arh2,buk1,buk2,busu,ko,kosc,os1,ospy,os2,sak,sash,sak2,sch3,sch4,scheep,bts}
 (and many others) were investigated the applications of selection principles in the
study of the properties of function spaces. In particular, the
properties of the space $C_p(X)$ were investigated. In this paper
we continue to study different selectors for sequences of dense
sets of $C_p(X)$.

\section{Main definitions and notation}

\bigskip

 We recall that a subset of $X$ that is the
 complete preimage of zero for a certain function from~$C(X)$ is called a zero-set.
A subset $O\subseteq X$  is called  a cozero-set (or functionally
open) of $X$ if $X\setminus O$ is a zero-set.

\medskip
Recall that the $i$-weight $iw(X)$ of a space $X$ is the smallest
infinite cardinal number $\tau$ such that $X$ can be mapped by a
one-to-one continuous mapping onto a Tychonoff space of the weight
not greater than $\tau$.

\medskip

\medskip

\begin{theorem} (Noble \cite{nob}) \label{th31}  A space $C_{p}(X)$ is separable iff
$iw(X)=\aleph_0$.
\end{theorem}

\medskip

Let $X$ be a topological space, and $x\in X$. A subset $A$ of $X$
{\it converges} to $x$, $x=\lim A$, if $A$ is infinite, $x\notin
A$, and for each neighborhood $U$ of $x$, $A\setminus U$ is
finite. Consider the following collection:

$\bullet$ $\Omega_x=\{A\subseteq X : x\in \overline{A}\setminus
A\}$;

$\bullet$ $\Gamma_x=\{A\subseteq X : x=\lim A\}$.

Note that if $A\in \Gamma_x$, then there exists $\{a_n\}\subset A$
converging to $x$. So, simply $\Gamma_x$ may be the set of
non-trivial convergent sequences to $x$.

\bigskip

We write $\Pi (\mathcal{A}_x, \mathcal{B}_x)$ without specifying
$x$, we mean $(\forall x) \Pi (\mathcal{A}_x, \mathcal{B}_x)$.

 So we have three types of topological properties
described through the selection principles:

$\bullet$  local properties of the form $S_*(\Phi_x,\Psi_x)$;

$\bullet$  global properties of the form $S_*(\Phi,\Psi)$;

$\bullet$  semi-local of the form $S_*(\Phi,\Psi_x)$.

\medskip

In paper \cite{os2}, we investigated different selectors for
sequences of dense sets of $C_p(X)$. We gave the characteristics
of selection principles $S_{1}(\mathcal{P},\mathcal{Q})$,
$S_{fin}(\mathcal{P},\mathcal{Q})$  for
$\mathcal{P},\mathcal{Q}\in \{\mathcal{D}$, $\mathcal{S}\}$
through the selection principles of a space $X$.

So for some selectors for sequences of dense sets of $C_p(X)$ (
Fig.~2) we obtained the corresponding characteristics through the
selection principles of a space $X$ (see Fig.~3 for a metrizable
separable space $X$).

\begin{center}
\ingrw{90}{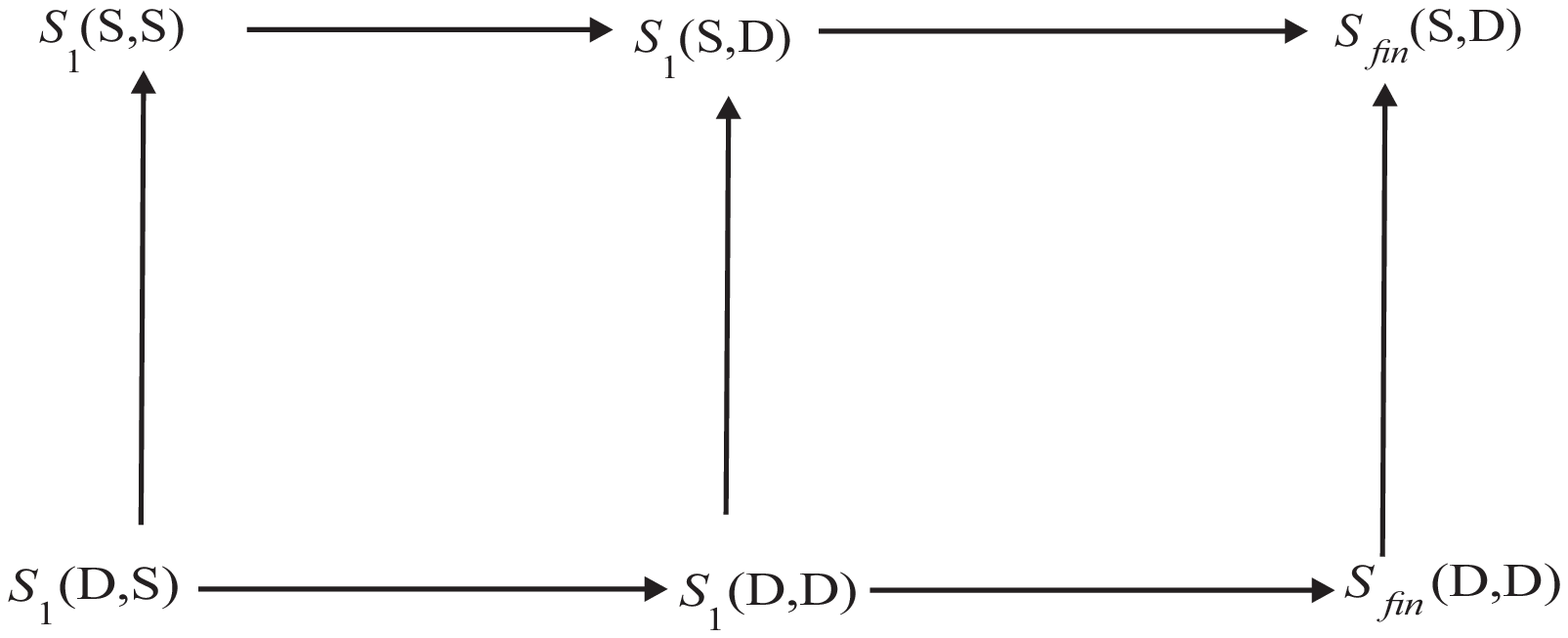}

\medskip

Fig.~2. The Diagram of selectors for sequences of dense sets of
$C_p(X)$.

\end{center}

\bigskip

\begin{center}
\ingrw{90}{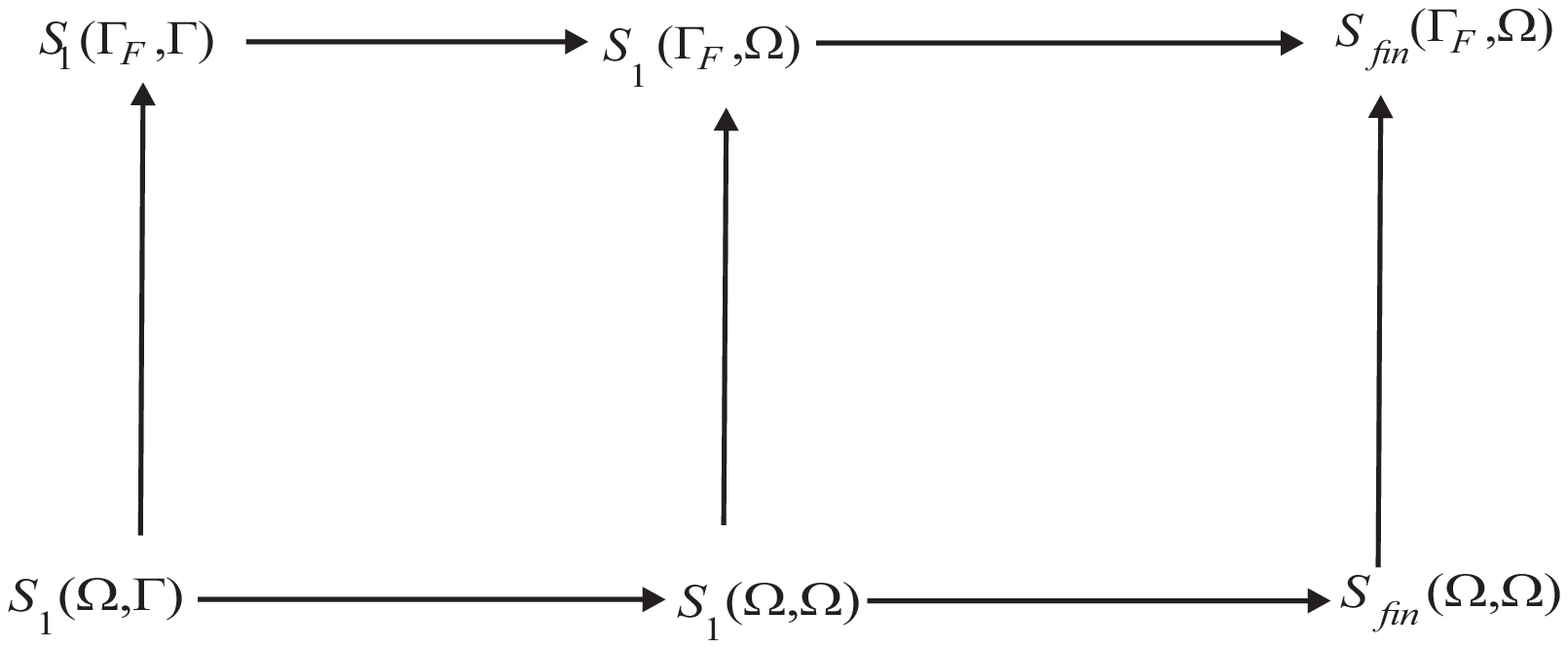}

\medskip

Fig.~3. The Diagram of selection principles for metrizable
separable space $X$ corresponding to selectors for sequences of
dense sets of $C_p(X)$.

\end{center}

\medskip

Our main goal is to describe the remaining topological properties
of $X$ of the Scheepers Diagram in terms of local, global and
semi-local properties of $C_p(X)$.

\section{The Rothberger and Menger properties}

\medskip
A space $X$ is said to be Rothberger \cite{rot} (or, \cite{mil})
if for every sequence $(\mathcal{U}_n : n\in \omega)$ of open
covers of $X$, there is a sequence $(V_n : n\in \mathbb{N})$ such
that for each $n$, $V_n\in \mathcal{U}_n$, and $\{V_n :
n\in\mathbb{N} \}$ is an open cover of $X$.

\begin{definition} A set $A\subseteq C_p(X)$ will
be called {\it $n$-dense} in $C_p(X)$, if for each $n$-finite set
$\{x_1,...,x_n\}\subset X$ such that $x_i\neq x_j$ for $i\neq j$
and an open sets $W_1,..., W_n$ in $\mathbb{R}$ there is $f\in A$
such that $f(x_i)\in W_i$ for $i\in \overline{1,n}$.

\end{definition}

Obviously, that if $A$ is a $n$-dense set of $C_p(X)$ for each
$n\in \mathbb{N}$ then $A$ is a dense set of $C_p(X)$.

For a space $C_p(X)$ we denote:

$\mathcal{A}_n$
--- the family of a $n$-dense subsets of $C_p(X)$.

If $n=1$, then we denote $\mathcal{A}$ instead of $\mathcal{A}_1$.

\begin{definition} Let $f\in C(X)$. A set $B\subseteq C_p(X)$ will
be called {\it $n$-dense} at point $f$, if for each $n$-finite set
$\{x_1,...,x_n\}\subset X$ and $\epsilon>0$ there is $h\in B$ such
that $h(x_i)\in (f(x_i)-\epsilon, f(x_i)+\epsilon)$ for $i\in
\overline{1,n}$.
\end{definition}

Obviously, that if $B$ is a $n$-dense at point $f$ for each $n\in
\mathbb{N}$ then $f\in \overline{B}$.

For a space $C_p(X)$ we denote:

$\mathcal{A}_{n,f}$
--- the family of a $n$-dense at point $f$ subsets of $C_p(X)$.

If $n=1$, then we denote $\mathcal{A}_f$ instead of
$\mathcal{A}_{1,f}$.

 Let $\mathcal{U}$ be an open cover of $X$ and $n\in \mathbb{N}$.

$\bullet$ $\mathcal{U}$ is an $n$-cover of $X$ if for each
$F\subset X$ with $|F|\leq n$, there is $U\in \mathcal{U}$ such
that $F \subset U$ \cite{bts1}.

Denote by $\mathcal{O}_n$ --- the family of open $n$-covers of
$X$.

$\bullet$  $S_{1}(\mathcal{O}, \mathcal{O})=S_{1}(\Omega,
\mathcal{O})$ \cite{sch3}.

$\bullet$ $S_{1}(\Omega,
\mathcal{O})=S_{1}(\{\mathcal{O}_n\}_{n\in \mathbb{N}},
\mathcal{O})$ \cite{bts1}.

\begin{theorem}\label{th143} For a   space $X$, the following statements are
equivalent:

\begin{enumerate}

\item  $C_p(X)$ satisfies $S_{1}(\mathcal{A},\mathcal{A})$;

\item $X$ satisfies $S_{1}(\mathcal{O}, \mathcal{O})$ [Rothberger
property];

\item $C_p(X)$ satisfies $S_{1}(\mathcal{A}_f,\mathcal{A}_f)$;

\item  $C_p(X)$ satisfies $S_{1}(\mathcal{A},\mathcal{A}_f)$;

\item  $C_p(X)$ satisfies $S_{1}(\mathcal{D},\mathcal{A})$;

\item  $C_p(X)$ satisfies $S_{1}(\{\mathcal{A}_n\}_{n\in
\mathbb{N}},\mathcal{A})$;

\item $C_p(X)$ satisfies $S_{1}(\{\mathcal{A}_{n,f}\}_{n\in
\mathbb{N}},\mathcal{A}_f)$;

\item  $C_p(X)$ satisfies $S_{1}(\{\mathcal{A}_n\}_{n\in
\mathbb{N}},\mathcal{A}_f)$.

\end{enumerate}

\end{theorem}

\begin{proof} $(1)\Rightarrow(2)$.  Let $\{\mathcal{O}_n\}_{n\in \mathbb{N}}$ be a
sequence of open covers of $X$. We set $A_n=\{f\in C(X):
f\upharpoonright (X\setminus U)=1$ and $f\upharpoonright K=q$ for
some $U\in \mathcal{O}_n$ , a finite set $K\subset U$ and $q\in
\mathbb{Q}\}$. It is not difficult to see that each $A_n$ is a
$1$-dense subset of $C_p(X)$  because  $\mathcal{O}_n$ is a cover
of $X$ and $X$ is Tychonoff.

 By the assumption there exists $f_n\in A_n$ such that
$\{f_n : n\in \mathbb{N}\}\in \mathcal{A}$.

 For each $f_n$ we
take $U_n\in \mathcal{O}_n$ such that
$f_n\upharpoonright(X\setminus U_n)=1$.

 Set $\mathcal{U}=\{ U_n : n\in \mathbb{N}\}$. For $x\in X$ we consider the basic open
 neighborhood $[x, W]$
of $\bf{0}$, where $W=(-\frac{1}{2},\frac{1}{2})$.

 Note that there is $m\in \mathbb{N}$ such that
$[x, W]$ contains $f_m\in \{f_n : n\in \mathbb{N}\}$. This means
$x\in
 U_m$. Hence $\mathcal{U}$ is a cover
of $X$.

$(2)\Rightarrow(3)$. Let $B_n\in \mathcal{A}_f$ for each $n\in
\mathbb{N}$. We renumber $\{B_n\}_{n\in \mathbb{N}}$ as
$\{B_{i,j}\}_{i,j\in \mathbb{N}}$.  Since $C(X)$ is homogeneous,
we may think that $f=\bf{0}$.  We set
$\mathcal{U}_{i,j}=\{g^{-1}(-1/i, 1/i) : g\in B_{i,j}\}$ for each
$i,j\in \mathbb{N}$. Since $B_{i,j}\in \mathcal{A}_0$,
$\mathcal{U}_{i,j}$ is an open cover of $X$ for each $i,j\in
\mathbb{N}$. In case the set $M=\{i\in \mathbb{N}: X\in
\mathcal{U}_{i,j} \}$ is infinite, choose $g_{m}\in B_{m,j}$ $m\in
M$ so that $g^{-1}(-1/m, 1/m)=X$, then $\{g_m : m\in
\mathbb{N}\}\in \mathcal{A}_f$.

So we may assume that there exists $i'\in \mathbb{N}$ such that
for each $i\geq i'$ and $g\in B_{i,j}$ $g^{-1}(-1/i, 1/i)$ is not
$X$.

For the sequence $\mathcal{V}_i=(\mathcal{U}_{i,j} : j\in
\mathbb{N})$ of open covers there exist $f_{i,j}\in B_{i,j}$ such
that $\mathcal{U}_i=\{f^{-1}_{i,j}(-1/i,1/i):  j\in \mathbb{N}\}$
is a cover of $X$.  Let $[x, W]$ be any basic open neighborhood of
$\bf{0}$, where $W=(-\epsilon, \epsilon)$, $\epsilon>0$. There
exists $m\geq i'$ and $j\in \mathbb{N}$  such that $1/m<\epsilon$
and $x\in f^{-1}_{m,j}(-1/m, 1/m)$. This means $\{f_{i,j}: i,j\in
\mathbb{N}\}\in \mathcal{A}_f$.

$(3)\Rightarrow(4)$ is immediate.

$(4)\Rightarrow(1)$. Let $A_n\in \mathcal{A}$ for each $n\in
\mathbb{N}$. We renumber $\{A_n\}_{n\in \mathbb{N}}$ as
$\{A_{i,j}\}_{i,j\in \mathbb{N}}$. Renumber the rational numbers
$\mathbb{Q}$ as $\{q_i : i\in \mathbb{N}\}$.  Fix
$i\in\mathbb{N}$. By the assumption there exists $f_{i,j}\in
A_{i,j}$ such that $\{f_{i,j} : j\in \mathbb{N}\}\in
\mathcal{A}_{q_i}$ where $q_i$  is the constant function to $q_i$.
Then $\{f_{i,j} : i,j\in \mathbb{N}\}\in \mathcal{A}$.

$(1)\Rightarrow(5)$. Since a dense set of $C_p(X)$ is a $1$-dense
set of $C_p(X)$, we have $C_p(X)$ satisfies
$S_{1}(\mathcal{D},\mathcal{A})$.

$(5)\Rightarrow(6)$. Let $D_n \in \mathcal{A}_n$ for each $n\in
\mathbb{N}$. We renumber $\{D_n\}_{n\in \mathbb{N}}$ as
$\{D_{i,j}\}_{i,j\in \mathbb{N}}$. Then $P_j=\{D_{i,j} : i\in
\mathbb{N}\}$ is a dense subset of $C_p(X)$ for each $j\in
\mathbb{N}$. By (5), there is $p_j\in P_j$ for each $j\in
\mathbb{N}$ such that $\{p_j: j\in \mathbb{N}\}\in \mathcal{A}$.
Hence, we have $C_p(X)$ satisfies $S_{1}(\{\mathcal{A}_n\}_{n\in
\mathbb{N}},\mathcal{A})$.

$(6)\Rightarrow(8)$ is immediate.

$(8)\Rightarrow(2)$. Claim that $X$ satisfies
$S_{1}(\{\mathcal{O}_n\}_{n\in \mathbb{N}}, \mathcal{O})$. Fix
$\{\mathcal{O}_n\}_{n\in \mathbb{N}}$. For every $n\in \mathbb{N}$
a set $\mathcal{S}_n=\{f\in C(X) : f\upharpoonright (X\setminus
U)=1$ and $f(x_i)\in \mathbb{Q}$ for each $i=\overline{1,n}$ for
$U\in \mathcal{O}_n$ and a finite set $K=\{x_1,..., x_n\}\subset U
\}$. Note that $\mathcal{S}_n\in \mathcal{A}_n$ for each $n\in
\mathbb{N}$. By $(8)$, there is $f_n\in \mathcal{S}_n$ for each
$n\in \mathbb{N}$ such that $\{f_n: n\in \mathbb{N}\}\in
\mathcal{A}_0$. Then $\{U_n : n\in \mathbb{N} \}\in \mathcal{O}$.

$(3)\Rightarrow(7)$ is immediate.

$(7)\Rightarrow(2)$. The proof is analogous to proof of
implication $(8)\Rightarrow(2)$.

\end{proof}

\medskip
A space $X$ is said to be Menger \cite{hur} (or, \cite{sash}) if
for every sequence $(\mathcal{U}_n : n\in \mathbb{N})$ of open
covers of $X$, there are finite subfamilies $\mathcal{V}_n\subset
\mathcal{U}_n$ such that $\bigcup \{\mathcal{V}_n : n\in
\mathbb{N} \}$ is a cover of $X$.

 Every $\sigma$-compact space is Menger,
and a Menger space is Lindel$\ddot{o}$f.

\begin{theorem}\label{th144} For a   space $X$, the following statements are
equivalent:

\begin{enumerate}

\item  $C_p(X)$ satisfies $S_{fin}(\mathcal{A},\mathcal{A})$;

\item $X$ satisfies $S_{fin}(\mathcal{O}, \mathcal{O})$ [Menger
property];

\item $C_p(X)$ satisfies $S_{fin}(\mathcal{A}_f,\mathcal{A}_f)$;

\item  $C_p(X)$ satisfies $S_{fin}(\mathcal{A},\mathcal{A}_f)$;

\item  $C_p(X)$ satisfies $S_{fin}(\mathcal{D},\mathcal{A})$;

\item  $C_p(X)$ satisfies $S_{fin}(\{\mathcal{A}_n\}_{n\in
\mathbb{N}},\mathcal{A})$;

\item $C_p(X)$ satisfies $S_{fin}(\{\mathcal{A}_{n,f}\}_{n\in
\mathbb{N}},\mathcal{A}_f)$;

\item  $C_p(X)$ satisfies $S_{fin}(\{\mathcal{A}_n\}_{n\in
\mathbb{N}},\mathcal{A}_f)$.

\end{enumerate}

\end{theorem}

\begin{proof}

The proof is analogous to proof of Theorem \ref{th143}.

\end{proof}

\section{$S_{1}(\mathcal{S},\mathcal{A})$}

\begin{definition}(Sakai)
A $\gamma$-cover $\mathcal{U}$ of co-zero sets of $X$  is {\bf
$\gamma_F$-shrinkable} if there exists a $\gamma$-cover $\{F(U) :
U\in \mathcal{U}\}$ of zero-sets of $X$ with $F(U)\subset U$ for
every $U\in \mathcal{U}$.
\end{definition}

For a topological space $X$ we denote:

$\bullet$ $\Gamma_F$ --- the family of $\gamma_F$-shrinkable
$\gamma$-covers of $X$.

\medskip

\begin{lemma}(Lemma 6.5 in \cite{os2})\label{lemma} Let $\mathcal{U}=\{U_n:n\in \mathbb{N}\}$ be a
$\gamma_F$-shrinkable co-zero cover of a space $X$. Then the set
$S=\{f\in C(X): f\upharpoonright (X\setminus U_n)\equiv 1$ for
some $n\in \mathbb{N}\}$ is sequentially dense in $C_p(X)$.
\end{lemma}

\begin{theorem}\label{th173} For a  space $X$, the following statements are equivalent:

\begin{enumerate}

\item  $C_p(X)$ satisfies $S_{1}(\mathcal{S},\mathcal{A})$;

\item $X$ satisfies $S_{1}(\Gamma_F, \mathcal{O})$;

\item $C_p(X)$ satisfies $S_{1}(\Gamma_{\bf0},\mathcal{A}_f)$;

\item  $C_p(X)$ satisfies $S_{1}(\mathcal{S},\mathcal{A}_f)$.

\end{enumerate}

\end{theorem}

\begin{proof} $(1)\Rightarrow(2)$. Let $\{\mathcal{U}_i: i\in \mathbb{N}\}\subset \Gamma_F$, $\mathcal{U}_i=\{ U^{m}_i: m\in \mathbb{N}\}$
for each $i\in \mathbb{N}$. For each $i\in \mathbb{N}$ we consider
a set $\mathcal{S}_i=\{ f^m_i\in C(X) : f^m_i\upharpoonright
(X\setminus U^{m}_i)=1$ for $m \in \mathbb{N} \}$.

Since $\mathcal{U}_i$ is a $\gamma$-cover of cozero subsets of
$X$, then, by Lemma \ref{lemma}, $\mathcal{S}_i$ is a sequentially
dense subset of $C_p(X)$ for each $i\in \mathbb{N}$.

Since $C_p(X)$ satisfies $S_{1}(\mathcal{S},\mathcal{A})$, there
is a set $\{f^{m(i)}_{i}: i\in\mathbb{N}\}$ such that for each
$i$, $f^{m(i)}_{i}\in \mathcal{S}_i$, and $\{f^{m(i)}_{i}:
i\in\mathbb{N} \}$ is an element of $\mathcal{A}$.

Consider a set $\{U^{m(i)}_{i}: i\in \mathbb{N}\}$.

(a). $U^{m(i)}_{i}\in \mathcal{U}_{i}$.

(b). $\{U^{m(i)}_{i}: i\in \mathbb{N}\}$ is a cover of $X$.

Let $x\in X$ and $U=\langle {\bf{0}}, x, \frac{1}{2} \rangle$ be a
base neighborhood of $\bf{0}$, then there is
$f^{m(i)_{j_0}}_{i_{j_0}}\in U$ for some $j_0\in \mathbb{N}$. It
follows that $x\in  U^{m(i)_{j_0}}_{i_{j_0}}$. We thus get $X$
satisfies $S_{1}(\Gamma_F, \mathcal{O})$.

$(3)\Rightarrow(2)$. Let $\{\mathcal{U}_i\}\subset \Gamma_F$. For
each $i\in \mathbb{N}$ we consider the
 set $\mathcal{S}_i=\{ f\in C(X) : f\upharpoonright
F(U)=0$ and $f\upharpoonright (X\setminus U)=1$ for $U\in
\mathcal{U}_i \}$.

Since $\mathcal{F}_i=\{F(U): U\in \mathcal{U}_i\}$ is a
$\gamma$-cover of $X$, we have that $\mathcal{S}_i$ converge to
${\bf 0}$, i.e. $\mathcal{S}_i\in \Gamma_0$ for each $i\in
\mathbb{N}$.

Since  $C_p(X)$ satisfies $S_{1}(\Gamma_{\bf0},\mathcal{A}_f)$,
there is a sequence $\{f_{i}\}_{i\in\mathbb{N}}$ such that for
each $i$, $f_{i}\in \mathcal{S}_i$, and $\{f_{i} :
i\in\mathbb{N}\}\in \mathcal{A}_{\bf 0}$.

Consider $\mathcal{V}=\{U_i : U_i\in \mathcal{U}_i$ such that
$f_i\upharpoonright F(U_i)=0$ and $f_i\upharpoonright (X\setminus
U_i)=1\} $. Let $x\in X$ and $W=[x,(-1,1)]$ be a neighborhood of
$\bf{0}$, then there exists $i_0\in \mathbb{N}$ such that
$f_{i_0}\in W$ .

 It follows that
$x\in U_{i_0}$ and $\mathcal{V}\in \mathcal{O}$. We thus get $X$
satisfies $S_{1}(\Gamma_F, \mathcal{O})$.

$(2)\Rightarrow(3)$.  Fix $\{S_n : n\in \mathbb{N}\}\subset
\Gamma_0$. We renumber $\{S_n : n\in \mathbb{N}\}$ as $\{S_{i,j}:
i,j\in \mathbb{N}\}$.

For each $i,j\in \mathbb{N}$ and $f\in S_{i,j}$, we put
$U_{i,j,f}=\{x\in X : |f(x)|<\frac{1}{i+j}\}$, $Z_{i,j,f}=\{x\in X
: |f(x)|\leq\frac{1}{i+j+1}\}$.

Each $U_{i,j,f}$ (resp., $Z_{i,j,f}$) is a cozero-set (resp.,
zero-set) in $X$ with $Z_{i,j,f}\subset U_{i,j,f}$. Let
$\mathcal{U}_{i,j}=\{ U_{i,j,f} : f\in S_{i,j}\}$ and
$\mathcal{Z}_{i,j}=\{ Z_{i,j,f} : f\in S_{i,j}\}$. So without loss
of generality, we may assume $U_{i,j,f}\neq X$ for each $i,j\in
\mathbb{N}$ and $f\in S_{i,j}$.

We can easily check that the condition $S_{i,j}\in \Gamma_0$
implies that $\mathcal{Z}_{i,j}$ is a $\gamma$-cover of $X$.

Since  $X$ satisfies $S_{1}(\Gamma_F, \mathcal{O})$ for each $j\in
\mathbb{N}$ there is a sequence $\{U_{i,j,f_{i,j}} : i\in
\mathbb{N}\}$ such that for each $i$, $U_{i,j,f_{i,j}}\in
\mathcal{U}_{i,j}$, and $\{U_{i,j,f_{i,j}} : i\in \mathbb{N}\}\in
\mathcal{O}$. Claim that $\{f_{i,j}:i,j\in \mathbb{N}\}\in
\mathcal{A}_0$. Let $x\in X$, $\epsilon>0$, and $W=[x,(-\epsilon,
\epsilon)]$ be a base neighborhood of $\bf{0}$, then there exists
$j'\in \mathbb{N}$ such that $\frac{1}{1+j'}<\epsilon$. It follow
that there exists $i'$ such that $f_{i',j'}(x)\in (-\epsilon,
\epsilon)$. So $C_p(X)$ satisfies
$S_{1}(\Gamma_{\bf0},\mathcal{A}_f)$.

$(3)\Rightarrow(4)$ is immediate.

$(4)\Rightarrow(1)$. Let $S_n\in \mathcal{S}$ for each $n\in
\mathbb{N}$. We renumber $\{S_n\}_{n\in \mathbb{N}}$ as
$\{S_{i,j}\}_{i,j\in \mathbb{N}}$. Renumber the rational numbers
$\mathbb{Q}$ as $\{q_i : i\in \mathbb{N}\}$.  Fix
$i\in\mathbb{N}$. By the assumption there exist $f_{i,j}\in
S_{i,j}$ such that $\{f_{i,j} : j\in \mathbb{N}\}\in
\mathcal{A}_{q_i}$ where $q_i$  is the constant function to $q_i$.
Then $\{f_{i,j} : i,j\in \mathbb{N}\}\in \mathcal{A}$.

\end{proof}












\begin{corollary}\label{th173} Suppose that a space $X$ satisfies $S_{1}(\Gamma,
\mathcal{O})$. Then a space $C_p(X)$ satisfies
$S_{1}(\mathcal{S},\mathcal{A})$.
\end{corollary}











\bigskip

{\bf Acknowledgment.} I should like to Thanks to Boaz Tsaban for
useful discussions and recommendation of the publication of this
study.

\bigskip

\bibliographystyle{model1a-num-names}
\bibliography{<your-bib-database>}







\end{document}